\documentclass{amsart}
\usepackage{amssymb, amsmath, graphicx, mathrsfs, enumitem, hyperref, aliascnt, multirow, amsthm, verbatim, hyperref, aliascnt}
\usepackage[all]{xy}

\newcommand{\Q}{\mathbb{Q}}
\newcommand{\R}{\mathbb{R}}
\newcommand{\Z}{\mathbb{Z}}

\renewcommand{\H}{\mathbb{H}}
\newcommand{\SL}{\mathrm{SL}}
\newcommand{\GL}{\mathrm{GL}}
\newcommand{\SO}{\mathrm{SO}}
\newcommand{\PSL}{\mathrm{PSL}}
\newcommand{\PGL}{\mathrm{PGL}}

\newcommand{\Isom}{\mathrm{Isom}}

\usepackage{color} 

\newtheorem{theorem}{Theorem}[section]

\newaliascnt{lemma}{theorem}
\newtheorem{lemma}[lemma]{Lemma}
\aliascntresetthe{lemma}

\newaliascnt{cor}{theorem}
\newtheorem{cor}[cor]{Corollary}
\aliascntresetthe{cor}

\newaliascnt{prop}{theorem}
\newtheorem{prop}[prop]{Proposition}
\aliascntresetthe{prop}

\newaliascnt{con}{theorem}

\aliascntresetthe{con}

\newaliascnt{claim}{theorem}

\aliascntresetthe{claim}

\theoremstyle{remark}
\newaliascnt{remark}{theorem}
\newtheorem{remark}[remark]{Remark}
\aliascntresetthe{remark}

\newaliascnt{defn}{theorem}

\aliascntresetthe{defn}

\newaliascnt{question}{remark}
\newtheorem{question}[question]{Question}
\aliascntresetthe{question}

\numberwithin{equation}{section}

\begin{document}

\title[Special subgroups of Bianchi groups]{Special subgroups of Bianchi groups}

\author{Michelle Chu}
\address{The University of Texas at Austin}
\email{mchu@math.utexas.edu}

\begin{abstract} We determine C-special subgroups of the Bianchi groups of index bounded above by 120 by effectivising the arguments of Agol-Long-Reid. These subgroups are congruence of level 2 or 4 and retract to the free group on two generators.  As a consequence, we find a C-special 20-sheeted cover of the figure-eight knot complement. We also determine C-special congruence subgroups for a family of cocompact arithmetic Kleinian groups.
\end{abstract}

\maketitle

\section{Introduction}
A 3-manifold $M$ is said to have a \emph{virtual} property if some finite-sheeted cover of $M$ has that property. Similarly, a group $\Gamma$ has a \emph{virtual} property if some finite-index subgroup of $\Gamma$ has that property.
It was shown in \cite{cooperlongreid} that non-compact finite volume 3-manifolds virtually contain closed embedded essential surfaces.  
More recent progress in 3-manifold theory has determined the virtual Haken and the virtual fibering conjectures for all finite volume hyperbolic 3-manifolds, asserting that every finite volume hyperbolic 3-manifold virtually contains an embedded essential surface, and furthermore, is virtually a surface bundle over the circle.
These results are implied by a stronger theorem which states that the fundamental groups of finite volume 3-manifolds are virtually special (see \cite{virtual_haken} for the closed case and \cite{wise_manuscript} or \cite{groves_manning} for the cusped case).
A group is \emph{special} as in \cite{haglund_wise} if it embeds in a right-angled Artin group (\emph{A-special}) or in a right-angled Coxeter group (\emph{C-special}). Every right-angled Artin group (RAAG) is a finite-index subgroup of a right-angled Coxeter group (RACG), so every special group embeds in a RACG \cite{MR1783167}.
Virtually special groups inherit many nice properties from the RACG. In particular, for fundamental groups of hyperbolic 3-manifolds, virtually special (together with tameness \cite{Agol_tameness,tame} and Canary's covering theorem \cite{canary}) implies LERF and both the virtual Haken and the virtual fibering conjectures.

The goal of this note is to address the following question:

\begin{question} \label{ques: main}
Given a virtually special group, can one determine a finite-index special subgroup? Can one bound its index?
\end{question}

Prior to Agol's and Wise's results \cite{virtual_haken,wise_manuscript}, virtually special was known for several classes 3-manifold groups. Agol, Long, and Reid showed in \cite{ALR} that the Bianchi groups are virtually C-special. Later in \cite{BHW}, Bergeron, Haglund, and Wise showed that the fundamental group of an arithmetic hyperbolic manifold of simplest type is virtually C-special.
Recently, \autoref{ques: main} was answered for the Seifert-Weber dodecahedral space, which is an arithmetic hyperbolic manifold of simplest type. Spreer and Tillmann constructed in \cite{STDodecahedra} an A-special cover of the Seifert-Weber dodecahedral space with degree 60. In this paper we algebraically construct C-special covers of the Bianchi groups.

\begin{theorem}\label{thm: main} Let $m$ be a square-free positive integer and $\mathcal{O}_m$ the ring of integers in the quadratid imaginary field $\Q(\sqrt{-m})$.
The Bianchi group $\PSL(2,\mathcal{O}_m)$ contains a subgroup $\Delta_m$ which embeds in a RACG and has index
\begin{equation*}
[\PSL(2,\mathcal{O}_m): \Delta_m]=
\begin{cases}
	48 & \text{if }m\equiv 1,2 \mod (4) \\
	120 & \text{if }m\equiv 3 \mod (8) \\
	72 & \text{if }m\equiv 7 \mod (8)
\end{cases}
\end{equation*}
where $\Delta_m$ is a principal congruence subgroup of level 2 if $m\equiv 1,2 \mod (4)$ and is otherwise congruence of level 4. 
\end{theorem}

This result relies on making effective the strategy exploited by Agol, Long, and Reid in \cite{ALR}. The main idea is to realize quadratic forms associated to the Bianchi groups as sub-forms of the standard form of signature $(6,1)$. The orthogonal group $O(6,1;\Z)$ contains the reflection group of a right-angled 6-dimensional hyperbolic Coxeter polytope. In fact, each $\H^3/\Delta_m$ immerses totally geodesically in this reflection orbifold.

A particular subgroup of a Bianchi group which has attracted much interest is the fundamental group of the figure-eight knot complement. The holonomy representation of the hyperbolic structure is given by \cite{Riley_quadpargp}
\begin{equation*}
\Gamma_8:=\left\langle \begin{pmatrix} 1&1\\0&1 \end{pmatrix},\begin{pmatrix} 1&0\\ \frac{1+\sqrt{-3}}{2}&1 \end{pmatrix} \right\rangle \subset \PSL(2,\mathcal{O}_3) \text{ of index 12.}
\end{equation*}
It can be shown that the intersection $\Gamma_8\cap \Delta_3$ is a subgroup of $\Gamma_8$ of index 20 and we get the following corollary. 

\begin{cor}
The figure-eight knot complement has a special finite-sheeted cover of degree 20.
\end{cor}

Another consequence of the embedding of $\Delta_m$ in a RACG is that $\Delta_m$ is virtually RFRS \cite{RFRS} (in fact, contains a RFRS subgroup of index $2^{26}$) and virtually retracts to its geometrically finite subgroups \cite{virt_retracts}.

\begin{prop}\label{prop: retract}
The subgroup $\Delta_m$ retracts onto the free group on 2 generators.
\end{prop}

The paper is organized as follows: In \autoref{sec: prelim} we give some preliminaries on orthogonal groups, congruence subgroups, and right-angled Coxeter groups. We introduce Bianchi groups and their associated quadratic forms in \autoref{sec: Bianchi quad form}. The proofs of \autoref{thm: main} and \autoref{prop: retract} are in \autoref{sec: pf of main}. Finally, we describe some cocompact examples in \autoref{sec: cocompact}.

\section{Preliminaries}\label{sec: prelim}

\subsection{Orthogonal groups of quadratic forms}
Let $f$ be a quadratic form with coefficients in a number field $K$. Let $S_f$ be the symmetric matrix associated to $f$. We say that two quadratic forms $f_1$ and $f_2$ are equivalent over $k$ is there exist some matrix $A\in\GL(n,K)$ such that $A^tS_{f_1}A=S_{f_2}$.

For $K$ a real number field, $\mathcal{O}_K$ its ring of integers, and $f$ a quadratic form over $K$ in $n+1$ variables, the Orthogonal group of $f$ is the group
$$ \mathrm{O}(f)=\{M\in\GL(n+1,\R)\ | M^t S_Q M=S_Q \}. $$

Now fix $K$ a totally real number field. Consider a quadratic form $Q$ defined on a vector space $V$ over $K$ of dimension $n+1$ such that $Q$ has signature $(n,1)$ but for every non-identity embedding $\sigma:k\hookrightarrow\R$, $Q^\sigma$ is positive-definite.
Let $\R^{n,1}$ be the $(n+1)$-dimensional vector space  with symmetric bilinear form
\begin{equation*} (u,v) = \frac{1}{2} \left( Q(u+v) - Q(u) - Q(v) \right) \end{equation*}
and symmetric matrix $S_Q$ (the symmetric matrix in $\GL(n+1,\R)$ with $Q(v)=v^t S_Q v$).

Let $\{x_0,x_1,\dots,x_n\}$ be an orthogonal basis for $Q$ with $Q(x_0)<0$ and $Q(x_i)>0$ for each $i\geq 1$. The vectors $v\in\R^{n,1}$ with $Q(v)=1$ form an $n-$dimensional hyperboloid $\mathcal{C}$ consisting of a positive sheet $\mathcal{C}^+=\{v\in \mathcal{C} | v_0>0\}$ and a negative sheet $\mathcal{C}^-=\{v\in \mathcal{C} | v_0<0\}$. The \emph{hyperboloid model} of hyperbolic space $\H^n$ is identified with $S^+$. The isometires of $\H^n$ are the orthogonal transformations of $\R^{n,1}$.

Then $\mathrm{O}(Q;\R)$ is the isometry group which preserves $\mathcal{C}$. The index two subgroup which preserves the positive sheet $\mathcal{C}^+$ is called $\mathrm{O}^+(Q;\R)$ (i.e. preserves the sign of $v_0$). This subgroup $\mathrm{O}^+(Q;\R)$ is identified with the full group of isometries of $\H^n$ and has two connected components, whose elements either preserve or reverse orientation. The orientation preserving isometry group of $\H^n$ is identified with the component of matrices with determinant 1, $\SO^+(Q;\R)$.

The subgroup $\mathrm{O}^+(Q,\mathcal{O}_K)=\mathrm{O}^+(Q,\R)\cap \GL(n+1,\mathcal{O}_K)$ (or $\SO^+(Q,\mathcal{O}_K)=\SO^+(Q,\R)\cap \GL(n+1,\mathcal{O}_K)$) of integral isometries is a lattice, discrete subgroup of finite covolume, in $\mathrm{O}(Q;\R)$ (or $\SO(Q;\R)$). Any lattice commensurable with such an $\mathrm{O}(Q;\mathcal{O}_K)$ or $\SO(Q;\mathcal{O}_K)$ is called an \emph{arithmetic lattice of simplest type}. 

\subsection{Congruence subgroups}
If $\mathscr{I}$ is an ideal in a ring $R$ and $\Delta\leq\SL(n,R)$, the principal congruence subgroup of level $(\mathscr{I})$ in $\Delta$ is the subgroup of matrices $M\in\Delta$ which can be written as $M=I+\mathscr{I}A$ where $A\in \GL(n,R)$. We denote this subgroup by $\Delta_{(\mathscr{I})}$.

\subsection{Right-angled hyperbolic polytope}
\label{subsec: the RACG}
A \emph{polytope} in $n$-dimensional hyperbolic space $\H^n$ is a convex polyhedron with finitely many actual and ideal vertices which is the convex hull of its vertices. A polytope is a \emph{Coxeter polytope} if the dihedral angle between any two adjacent sides is either $0$ or $\pi/k$ for some integer $k\geq 2$. A Coxeter polytope is \emph{right-angled} if all the dihedral angles are either $0$ or $\pi/2$.

The group generated by the reflections in the sides of a Coxeter polytope is a hyperbolic Coxeter group and is a discrete subgroup of $\Isom(\H^n)$. The group generated by reflections in the sides of a right-angled polytope is also right-angled Coxeter group, denoted \emph{RACG}.

Let $F_n$ be the diagonal quadratic form 
\begin{equation} F_n:=-x_0^2+ x_1^2 + \dots + x_n^2. 
\end{equation}

\begin{theorem}[{\cite[Theorem 2.1]{ERT}}]
For $2\leq n \leq 7$, $\mathrm{O}^+(F_n;\Z)_{(2)}$ is a RACG. It is the reflection group of an all-right hyperbolic polytope of dimension $n$, and $\SO^+(F_n;\Z)_{(2)}$ is its index 2 subgroup of orientation preserving isometries.
\label{th: RACG}
\end{theorem}

\section{Bianchi groups and quadratic forms}\label{sec: Bianchi quad form}
The Bianchi groups are the arithmetic Kleinian groups $\PSL(2,\mathcal{O}_m)$ where $m$ is a positive square-free integer and 
\begin{equation}
\mathcal{O}_m = \begin{cases}\Z[\sqrt{-m}] & m\equiv1,2\mod(4)\\ \Z[\frac{1+\sqrt{-m}}{2}] & m\equiv3\mod(4)\end{cases}
\end{equation}
is the ring of integers in the field $\Q(\sqrt{-m})$. Any non-cocompact arithmetic lattice in $\Isom^+(\H^3)$ is commensurable to some Bianchi group (see e.g. \cite[\S8.2]{MR_book}).

\subsection{The Bianchi groups and quadratic forms} \label{subsec: bianchi quad form}
We now give a precise relationship between Bianchi groups and orthogonal groups of quadratic forms following \cite[\S 2]{JM} and \cite[\S 1.3]{EGM}.

For $m$ square-free integer, define the quadratic forms
\begin{equation}
Q_m(x_0,x_1,x_2,x_3) :=
\begin{cases}
 2x_0 x_1 + 2x_2^2 + 2m x_3^2 & \text{ if }m\equiv 1,2 \mod 4 \\
 2x_0 x_1 + 2x_2^2 + 2x_2 x_3 + \frac{m+1}{2} x_3^2 & \text{ if }m\equiv 3 \mod 4
\end{cases}
\end{equation}
with corresponding symmetric matrices
\begin{equation}\label{eq: mat S_m}
S_{m_{1,2}}=
\begin{pmatrix}
0 & 1 & 0 & 0\\
1 & 0 & 0 & 0\\
0 & 0 & 2 & 0 \\
0 & 0 & 0 & 2m
\end{pmatrix} 
\text{ and }
S_{m_{3}}=
\begin{pmatrix}
0 & 1 & 0 & 0\\
1 & 0 & 0 & 0\\
0 & 0 & 2 & 1 \\
0 & 0 & 1 & \frac{m+1}{2}
\end{pmatrix} 
.
\end{equation}

For $m\equiv 1,2\mod4$, define the homomorphism 
\begin{equation*}
\varphi_m:\PGL(2,\mathcal{O}_m)\rightarrow \SO(Q_m;\Z)
\end{equation*}
by sending
$\alpha=\begin{pmatrix} 
a_0+a_1\sqrt{-m}&b_0+b_1\sqrt{-m}
\\ c_0+c_1\sqrt{-m}&d_0+d_1\sqrt{-m} \end{pmatrix}$
to $(\det(\alpha))^{-1}$ times 
\small{\begin{multline}\label{eq: varphi 1,2}
\left(\begin{matrix}  d_0^2+m d_1^2 & -c_0^2-m c_1^2 & 2 \left(c_0 d_0+m c_1 d_1\right)
\\ -b_0^2-m b_1^2 & a_0^2+m a_1^2 & -2 \left(a_0 b_0+m a_1 b_1\right) 
\\ b_0 d_0+m b_1 d_1 & -a_0 c_0-m a_1 c_1 & b_0 c_0+m b_1 c_1+a_0 d_0+m a_1 d_1 
\\ b_1 d_0-b_0 d_1 & a_0 c_1-a_1 c_0 & b_1 c_0-b_0 c_1+a_1 d_0-a_0 d_1 \end{matrix}\right.
\\ \left.
\begin{matrix}
-2 m \left(c_1 d_0-c_0 d_1\right)
\\ 2 m \left(a_1 b_0-a_0 b_1\right)
\\ m \left(b_1 c_0-b_0 c_1-a_1 d_0+a_0 d_1\right)
\\  -b_0 c_0-m b_1 c_1+a_0 d_0+m a_1 d_1
\end{matrix}
\right).
\end{multline}}\normalsize
The kernel of this homomorphism is $\Q^\times I$, so it restricts to an injection
$$\PSL(2,\mathcal{O}_m)\hookrightarrow\SO^+(Q_m;\Z)$$
whose image is the spinorial group \cite{JM}.

For $m\equiv 3\mod4$, let $m=4k-1$ and define the homomorphism 
\begin{equation*}
\varphi_m:\PGL(2,\mathcal{O}_m)\rightarrow \SO(Q_m;\Z)
\end{equation*}
by sending
$\alpha=\begin{pmatrix} 
a_0+a_1\frac{1+\sqrt{-m}}{2}&b_0+b_1\frac{1+\sqrt{-m}}{2}
\\ c_0+c_1\frac{1+\sqrt{-m}}{2}&d_0+d_1\frac{1+\sqrt{-m}}{2} \end{pmatrix}$
to $(\det(\alpha))^{-1}$ times 
\small{\begin{multline}\label{eq: varphi 3}
\left(\begin{matrix}  d_0^2+d_1 d_0+k d_1^2 & -c_0^2-c_1 c_0-k c_1^2 & 2 c_0 d_0+c_1 d_0+c_0 d_1+2 k c_1 d_1 \\
 -b_0^2-b_1 b_0-k b_1^2 & a_0^2+a_1 a_0+k a_1^2 & -2 a_0 b_0-a_1 b_0-a_0 b_1-2 k a_1 b_1 \\
 b_0 d_0+b_1 d_0+k b_1 d_1 & -a_0 c_0-a_1 c_0-k a_1 c_1 &
\begin{smallmatrix}b_0 c_0+b_1 c_0+k b_1 c_1+a_0 d_0+a_1 d_0+k a_1 d_1\end{smallmatrix} \\
 b_0 d_1-b_1 d_0 & a_1 c_0-a_0 c_1 &-b_1 c_0+b_0 c_1-a_1 d_0+a_0 d_1 
\end{matrix}\right.
\\ \left.
\begin{matrix}
 c_0 d_0+2 k c_1 d_0-2 k c_0 d_1+c_0 d_1+k c_1 d_1
\\  -a_0 b_0-2 k a_1 b_0+2 k a_0 b_1-a_0 b_1-k a_1 b_1
\\  b_0 c_0-k b_1 c_0+b_1 c_0+k b_0 c_1+k b_1 c_1+k a_1d_0-k a_0 d_1
\\  -b_0 c_0-b_1 c_0-k b_1 c_1+a_0 d_0+a_0 d_1+k a_1 d_1
\end{matrix}
\right) .
\end{multline}}\normalsize
The kernel of this homomorphism is $\Q^\times I$, so it restricts to an injection $$\PSL(2,\mathcal{O}_m)\hookrightarrow\SO^+(Q_m;\Z).$$

\subsection{Congruence subgroups of Bianchi groups}\label{subsec: cong Bianchi}
If $\mathscr{I}$ is an ideal in $\mathcal{O}_m$, the quotient $\PSL(2,\mathcal{O}_m)/\PSL(2,\mathcal{O}_m)(\mathscr{I})$ of $\PSL(2,\mathcal{O}_m)$ by the principal congruence subgroup of level $\mathscr{I}$ is isomorphic to $\PSL(2,\mathcal{O}_m/\mathscr{I})$. The \emph{norm} of $\mathscr{I}$ is given by $N(\mathscr{I})=|\mathcal{O}_m/\mathscr{I}|$ and is multiplicative $N(\mathscr{IJ})=N(\mathscr{I})N(\mathscr{J})$.

This gives a formula for the index $\PSL(2,\mathcal{O}_m)(\mathscr{I})$ using a decomposition of $\mathscr{I}$ into powers of prime ideals $\mathscr{I}=\mathcal{P}_1^{j_1}\cdots\mathcal{P}_r^{j_r}$ (see \cite{Dickson}). We have
\begin{equation*}
[\PSL(2,\mathcal{O}_m):\PSL(2,\mathcal{O}_m)(\mathscr{I})]=\begin{cases}
6 & \text{when }N(\mathscr{I})=2\\
N(\mathscr{I})^3 \prod_{\mathcal{P}|\mathscr{I}} \left(1-\frac{1}{N(\mathscr{I})^2}\right) & \text{when } 2\in \mathscr{I}\\
\frac{N(\mathscr{I})^3}{2} \prod_{\mathcal{P}|\mathscr{I}} \left(1-\frac{1}{N(\mathscr{I})^2}\right) & \text{otherwise }.
\end{cases}
\end{equation*}

For a rational prime $p$ there are three possibilities for the decomposition of the ideal $p\mathcal{O}_m$ in $\mathcal{O}_m$:
\begin{equation*}
p\mathcal{O}_m=\begin{cases}
\mathcal{P}^2 & p \text{ is ramified and } N(\mathcal{P})=p \\
\mathcal{P} & p \text{ is inert and } N(\mathcal{P})=p^2 \\
\mathcal{P}_1\mathcal{P}_2 & p \text{ is split and } N(\mathcal{P}_1)=N(\mathcal{P}_2)=p .
\end{cases}
\end{equation*}

Therefore since 2 is ramified when $m\equiv1,2\mod4$, is inert when $m\equiv3\mod4$, and is split when $m\equiv7\mod4$ we have
\begin{equation}\label{eq: level 2 index}
[\PSL(2,\mathcal{O}_m):\PSL(2,\mathcal{O}_m)_{(2)}]=\begin{cases}
48 & \text{if }m\equiv1,2\mod4\\
60 & \text{if }m\equiv3\mod8\\
36 & \text{if }m\equiv7\mod8 .
\end{cases}
\end{equation}

\section{Proof of \autoref{thm: main}}\label{sec: pf of main}

\subsection{The special subgroup for \texorpdfstring{$m\equiv 1,2\mod 4$}{TEXT}}

By \autoref{eq: level 2 index} when $m\equiv 1,2\mod 4$ we have
\begin{equation}
\left[\PSL(2,\mathcal{O}_m) : \PSL(2,\mathcal{O}_m)_{(2)} \right] = 48.
\end{equation}

\begin{prop}\label{prop: only need level 2 1,2} Let $m\equiv 1,2\mod4$ be square-free positive. The principal congruence subgroup $\PSL(2,\Z[\sqrt{-m}])_{(2)}$ embeds in the RACG $\SO^+(F;\Z)_{(2)}$.
\end{prop}

\begin{proof}
Let $P_m:= Q_m \oplus \langle 2m,2m,2m \rangle$. The group $\SO^+(Q_m;\Z)$ is naturally a subgroup of $\SO^+(P_m;\Z)$.

By Lagrange's 4-square theorem, write $m= w^2 + x^2 + y^2 + z^2$, a sum of four squares.
Consider the $7\times 7$ matrices
\begin{equation}
A_{m} = 
\begin{pmatrix}
\frac{1}{2} & -\frac{1}{2} & 0 & 0 & 0 & 0 & 0 \\
\frac{1}{2} & \frac{1}{2} & 0 & 0 & 0 & 0 & 0 \\
0 & 0 & 1 & 0 & 0 & 0 & 0 \\
0 & 0 & 0 & w & -x & -y & -z \\
0 & 0 & 0 & x & w & z & -y \\
0 & 0 & 0 & y & -z & w & x \\
0 & 0 & 0 & z & y & -x & w \\
\end{pmatrix}
\label{eq: mat Am 1,2}\end{equation}
with inverse
\begin{equation}
A_{m}^{-1} = 
\begin{pmatrix}
1&1&0&0&0&0&0 \\
-1&1&0&0&0&0&0 \\
0&0&1&0&0&0&0 \\
0&0&0& \frac{w}{m} & \frac{x}{m} & \frac{y}{m} & \frac{z}{m} \\
0&0&0& -\frac{x}{m} & \frac{w}{m} & -\frac{z}{m} & \frac{y}{m} \\
0&0&0& -\frac{y}{m} & \frac{z}{m} & \frac{w}{m} & -\frac{x}{m} \\
0&0&0& -\frac{z}{m} & -\frac{y}{m} & \frac{x}{m} & \frac{w}{m} \\
\end{pmatrix}  .
\label{eq: mat AmI 1,2}\end{equation}

Let $S_F$ be the diagonal matrix associated to $F_6$ and $S_{P_m}$ the symmetric matrix associated to $\frac{1}{2}P_m$:
\begin{equation}
S_F = 
\begin{pmatrix}
-1 & 0 & 0 & 0 & 0 & 0 & 0 \\
0 & 1 & 0 & 0 & 0 & 0 & 0 \\
0 & 0 & 1 & 0 & 0 & 0 & 0 \\
0 & 0 & 0 & 1 & 0 & 0 & 0 \\
0 & 0 & 0 & 0 & 1 & 0 & 0 \\
0 & 0 & 0 & 0 & 0 & 1 & 0 \\
0 & 0 & 0 & 0 & 0 & 0 & 1 \\
\end{pmatrix}
\label{eq: mats F}\end{equation}
\begin{equation}
S_{P_m} = 
\begin{pmatrix}
0 & \frac{1}{2} & 0 & 0 & 0 & 0 & 0 \\
\frac{1}{2} & 0 & 0 & 0 & 0 & 0 & 0 \\
0 & 0 & 1 & 0 & 0 & 0 & 0 \\
0 & 0 & 0 & m & 0 & 0 & 0 \\
0 & 0 & 0 & 0 & m & 0 & 0 \\
0 & 0 & 0 & 0 & 0 & m & 0 \\
0 & 0 & 0 & 0 & 0 & 0 & m \\
\end{pmatrix}  .
\end{equation}
Then $A_{m}^t S_F A_{m}=S_{P_m}$. Since $A_{m}$ has determinant $-m^2/2$, it is in $\GL(7,\Q)$ and the forms $F_6$ and $P_m$ are equivalent over $\Q$ and thus $A_{m} \SO^+(P_m;\Q) A_{m}^{-1}=\SO^+(F_6;\Q)$. Therefore, $A_{m} \SO^+(Q_m;\Q) A_{m}^{-1}\subset \SO^+(F_6;\Q)$.

A matrix $N$ in $\PSL(2,\mathcal{O}_m)_{(2)}$ has form
\begin{equation*}
\begin{pmatrix} 2a_0+1+2a_1\sqrt{-m} & 2b_0+2b_1\sqrt{-m}\\ 2c_0+2c_1\sqrt{-m}& 2d_0+1+2d_1\sqrt{-m} \end{pmatrix} ,
\end{equation*}
and from \autoref{eq: varphi 1,2}, its image in $\varphi_m$ is given by
\small{\begin{multline}\label{eq: image other lv 2 at 1,2}
\left(\begin{matrix} 
4\left(d_0^2+d_0+md_1^2\right)+1 &-4\left(c_0^2+m c_1^2\right)
\\ -4\left(b_0^2+m b_1^2\right) & 4\left(a_0^2+a_0+m a_1^2\right)+1
\\ 2\left(2 d_0 b_0+b_0+2 m b_1 d_1\right) &-2\left(2a_0 c_0+c_0+2 m a_1 c_1\right)
\\ 2\left(2 d_0 b_1+b_1-2 b_0 d_1\right) & -2 \left(2 a_1 c_0-2 a_0 c_1-c_1\right)
\end{matrix} \right.
\\
\begin{matrix}
4 \left(2 d_0 c_0+c_0+2 m c_1 d_1\right) 
\\ -4 \left(2 a_0 b_0+b_0+2 m a_1 b_1\right) 
\\ 2 \left(2 d_0 a_0+a_0+2 b_0 c_0+2 m b_1 c_1+d_0+2 m a_1 d_1\right)+1 
\\ 2 \left(2 d_0 a_1+a_1+2 b_1 c_0-2 b_0 c_1-2 a_0 d_1-d_1\right) 
\end{matrix}
\\ \left.
\begin{matrix}
-4 m \left(2 d_0 c_1+c_1-2 c_0 d_1\right)
\\ 4 m \left(2 a_1 b_0-2 a_0 b_1-b_1\right) 
\\ 2m\left(-2 d_0 a_1-a_1+2 b_1 c_0-2 b_0 c_1+2 a_0 d_1+d_1\right)
\\ 1-2 \left(-2 d_0a_0-a_0+2 b_0 c_0+2 m b_1 c_1-d_0-2 m a_1 d_1\right) 
\end{matrix}
\right).
\end{multline}}\normalsize

Let $N'$ be
\begin{equation*}
\begin{pmatrix}
\varphi_m(N)& 0_{4\times3} \\
0_{3\times4} & I_{3\times3} \end{pmatrix} \in\SO^+(P_m;\Z) .
\end{equation*}
It is now an easy check to see that $A_{m}N'A_{m}^{-1}\in \SO^+(F_6;\Q)$ and in fact $A_{m}N'A_{m}^{-1}\equiv I\mod(2)$. So $A_{m}N'A_{m}^{-1}\in \SO^+(F;\Z)_{(2)}$.
\end{proof}

\begin{remark}\label{rmk: between level 4 and 2}
From \autoref{eq: image other lv 2 at 1,2} we see that 
\begin{equation*}
\varphi_m\left(\PSL(2,\mathcal{O}_m)\right)_{(4)}\leq \varphi_m\left(\PSL(2,\mathcal{O}_m)_{(2)}\right) \leq \varphi_m\left(\PSL(2,\mathcal{O}_m)\right)_{(2)} .
\end{equation*}
\end{remark}

\begin{remark}
In the ring of integers $\mathcal{O}_m$ for the number field $\Q(\sqrt{-m})$, $(2)$ is not always a prime ideal. One might ask whether a congruence subgroup with level a prime over 2 suffices to embed in a RACG. Unfortunately, using the methods above, a prime over 2 is not enough. Consider as an example the case of $m=1$. The ideal $(2)$ ramifies as $(1+i)^2$. However, the principal congruence subgroup $\PSL(2,\Z[i])_{(1+i)}$ is not contained in $\SO^+(F;\Z)_{(2)}$ via the map $\varphi_1$ and the conjugation by $A_{m_{1,2}}$. Indeed if we write $1=1^2+0+0+0$ (i.e. $w=1$, $x=y=z=0$ in the definition of $A_{m_{1,2}}$ in \autoref{eq: mat Am 1,2}), the element
\begin{equation*}
A_1^{-1}\varphi_1\left( \begin{pmatrix}
1&1+i \\ 0&1
\end{pmatrix} \right) A_1 
=
\begin{pmatrix}
2&1&1&1&0&0&0\\
-1&0&-1&-1&0&0&0\\
1&1&1&0&0&0&0\\
1&1&0&1&0&0&0\\
0&0&0&0&1&0&0\\
0&0&0&0&0&1&0\\
0&0&0&0&0&0&1
\end{pmatrix}
\end{equation*}
does not reduce to the identity modulo 2.
\end{remark}

\subsection{The special subgroup for \texorpdfstring{$m\equiv 3\mod 4$}{TEXT}}

\begin{prop}\label{prop: only need level 2 3} Let $m\equiv 3\mod4$ be square-free positive. The principal congruence subgroup $\PSL(2,\Z[\sqrt{-m}])_{(2)}$ has an index 2 subgroup $\Delta_m$ which embeds in the RACG $\SO^+(F;\Z)_{(2)}$ and with
\begin{equation}
\left[\PSL(2,\mathcal{O}_m) : \Delta_m \right] = \begin{cases}
120 & \text{if }m\equiv3\mod8
\\ 72 & \text{if }m\equiv7\mod8 .
\end{cases}
\end{equation}
\end{prop}

\begin{proof}
Let $P_m:= Q_m \oplus \langle 2m,2m,2m \rangle$. The group $\SO^+(Q_m;\Z)$ is naturally a subgroup of $\SO^+(P_m;\Z)$.

By Lagrange's 4-square theorem, write $m= w^2 + x^2 + y^2 + z^2$, a sum of four squares.
Consider the $7\times 7$ matrices
\begin{equation}
A_{m_{3}} = 
\begin{pmatrix}
\frac{1}{2} & -\frac{1}{2} & 0 & 0 & 0 & 0 & 0 \\
\frac{1}{2} & \frac{1}{2} & 0 & 0 & 0 & 0 & 0 \\
0 & 0 & 1 & \frac{1}{2} & 0 & 0 & 0 \\
0 & 0 & 0 & -\frac{w}{2} & -x & -y & -z \\
0 & 0 & 0 & -\frac{x}{2} & w & z & -y \\
0 & 0 & 0 & -\frac{y}{2} & -z & w & x \\
0 & 0 & 0 & -\frac{z}{2} & y & -x & w \\
\end{pmatrix}
\label{eq: mat Am 3}\end{equation}
with inverse
\begin{equation}
A_{m_{3}}^{-1} = 
\begin{pmatrix}
1&1&0&0&0&0&0 \\
-1&1&0&0&0&0&0 \\
0&0&1&\frac{w}{m}&\frac{x}{m}&\frac{y}{m}&\frac{z}{m} \\
0&0&0&-\frac{2w}{m}&-\frac{2x}{m}&-\frac{2y}{m}&-\frac{2z}{m} \\
0&0&0&-\frac{x}{m}&\frac{w}{m}&-\frac{z}{m}&\frac{y}{m} \\
0&0&0&-\frac{y}{m}&\frac{z}{m}&\frac{w}{m}&-\frac{x}{m} \\
0&0&0&-\frac{z}{m}&-\frac{y}{m}&\frac{x}{m}&\frac{w}{m} \\
\end{pmatrix} .
\label{eq: mat AmI 3}\end{equation}

Let $S_F$ be the diagonal matrix associated to $F_6$ (see \autoref{eq: mats F}) and $S_{P_m}$ the symmetric matrix associated to $\frac{1}{2}P_m$:
\begin{equation}
\begin{pmatrix}
0 & \frac{1}{2} & 0 & 0 & 0 & 0 & 0 \\
\frac{1}{2} & 0 & 0 & 0 & 0 & 0 & 0 \\
0 & 0 & 1 & \frac{1}{2} & 0 & 0 & 0 \\
0 & 0 & \frac{1}{2} & \frac{m+1}{4} & 0 & 0 & 0 \\
0 & 0 & 0 & 0 & m & 0 & 0 \\
0 & 0 & 0 & 0 & 0 & m & 0 \\
0 & 0 & 0 & 0 & 0 & 0 & m \\
\end{pmatrix}.
\end{equation}
Then $A_{m}^t S_F A_{m}=S_{P_m}$. Since $A_{m}$ has determinant $-m^2/4$, it is in $\GL(7,\Q)$ and the forms $F_6$ and $P_m$ are equivalent over $\Q$ and thus $A_{m} \SO^+(P_m;\Q) A_{m}^{-1}=\SO^+(F_6;\Q)$. Therefore, $A_{m} \SO^+(Q_m;\Q) A_{m}^{-1}\subset \SO^+(F_6;\Q)$.

A matrix $N$ in $\PSL(2,\mathcal{O}_m)_{(2)}$ has form
\begin{equation*}
\begin{pmatrix} 2a_0+1+2a_1\frac{1+\sqrt{-m}}{2} & 2b_0+2b_1\frac{1+\sqrt{-m}}{2}\\ 2c_0+2c_1\frac{1+\sqrt{-m}}{2}& 2d_0+1+2d_1\frac{1+\sqrt{-m}}{2}\end{pmatrix} ,
\end{equation*}
and from \autoref{eq: varphi 3}, its image in $\varphi_m$ is given by
\small{\begin{multline}\label{eq: image other lv 2 at 3}
\left(\begin{matrix} 
2d_1+4\left(d_0^2+d_1 d_0+d_0+k d_1^2\right)+1 & -4 \left(c_0^2+c_1 c_0+k c_1^2\right)
\\ -4 \left(b_0^2+b_1 b_0+k b_1^2\right) & 2a_1 + 4\left( a_0^2+ a_1 a_0+ a_0+ k a_1^2\right)+1
\\ 2(b_0+b_1)+4 \left(d_0 b_0+b_1 d_0+k b_1 d_1\right) & -2c_0-4 \left(a_0 c_0+a_1 c_0+k a_1 c_1\right)
\\  -2b_1-4 \left(d_0 b_1-b_0 d_1\right) & -2c_1 +4\left(a_1 c_0-a_0 c_1\right)
\end{matrix}\right.
\\
\begin{matrix}
2(c_1)+4\left(2d_0c_0+d_1c_0+c_0+c_1 d_0+2kc_1d_1\right)
\\ -2b_1-4\left(2a_0b_0+a_1b_0+b_0+a_0b_1+2k a_1b_1\right)
\\ 2(a_0+a_1+d_0)+ 4\left(d_0 a_0+b_0 c_0+b_1 c_0+k b_1c_1+a_1 d_0+k a_1 d_1\right)+1
\\ 2(d_1-a_1)-4\left(d_0 a_1+b_1 c_0-b_0 c_1-a_0 d_1\right)
\end{matrix}
\\ \left.
\begin{matrix}
2c_0+ 4 (d_0c_0-2kd_1c_0+d_1c_0+k c_1+2kc_1d_0+kc_1 d_1)
\\ -2(b_0+b_1)-4\left(a_0 b_0+2k a_1 b_0-k b_1-2k a_0 b_1+a_0 b_1+k a_1 b_1 \right) 
\\ 2(k a_1-k d_1) +4(b_0 c_0+b_1 c_0) +4k\left(d_0 a_1-b_1 c_0+b_0 c_1+b_1 c_1-a_0 d_1\right)
\\ 1+2(a_0+d_0+d_1)-4\left(-d_0a_0-d_1 a_0+b_0 c_0+b_1 c_0+k b_1 c_1-k a_1 d_1\right)
\end{matrix} \right) .
\end{multline}}\normalsize

Note that $N$ has determinant 1. This implies
 \begin{multline}\label{eq: determinant 3}
1+2a_0+a_1-4b_0c_0-2b_1c_0-2b_0c_1-2b_1c_1+2d_0+4a_0d_0+ 2a_1d_0\\+d_1+2a_0d_1+4b_1c_1k-4a_1d_1k=1
\end{multline}
\begin{equation*}
\text{ and } a_1-2b_1c_0-2b_0c_1-2b_1c_1+2a_1d_0+d_1+2a_0d_1+2a_1d_1 =0
\end{equation*}
so $a_1\equiv d_1\mod 2$. 

Let $N'$ be
\begin{equation*}
\begin{pmatrix}
\varphi_m(N)& 0_{4\times3} \\
0_{3\times4} & I_{3\times3} \end{pmatrix} \in\SO^+(P_m;\Z).
\end{equation*}
It is now also true as in the case of $m\equiv1,2\mod4$ in \autoref{prop: only need level 2 1,2} that $A_{m}N'A_{m}^{-1}\in \SO^+(F_6;\Q)$. However, it requires some clever substitutions using \autoref{eq: determinant 3}.

Unfortunately, it is not always the case that $A_{m}N'A_{m}^{-1}\equiv I\mod(2)$, indeed 
\begin{equation}
A_{m}N'A_{m}^{-1}\equiv \begin{pmatrix}
1&0&b_1+c_1&w\beta&x\beta&y\beta&z\beta 
\\ 0&1&\beta&w\beta&x\beta&y\beta&z\beta
\\ \beta&\beta&1&0&0&0&0 
\\ w\beta& w\beta&0&1&0&0&0 
\\ x\beta& x\beta&0&0&1&0&0 
\\ y\beta& y\beta&0&0&0&1&0 
\\ z\beta& z\beta&0&0&0&0&1
\end{pmatrix}   \mod 2
\end{equation}
where $\beta=b_1+c_1$. However, consider the subgroup of $\PSL(2,\mathcal{O}_m)_{(2)}$ of index 2 given by
\begin{equation}\label{eq: delta m3m4}
\Delta_m:=\left\{ \begin{pmatrix} 2a_0+1+2a_1\frac{1+\sqrt{-m}}{2} & 2b_0+2b_1\frac{1+\sqrt{-m}}{2}\\ 2c_0+2c_1\frac{1+\sqrt{-m}}{2}& 2d_0+1+2d_1\frac{1+\sqrt{-m}}{2}\end{pmatrix}: b_1\equiv c_1\mod2 \right\} .
\end{equation}
Then for $N\in\Delta_m$ $A_{m_i}N'A_{m_i}^{-1}\in \SO^+(F;\Z)_{(2)}$ and the proposition follows from \autoref{eq: level 2 index}.
\end{proof}

\begin{remark}\label{rmk: between level 4 and 3}
From \autoref{eq: image other lv 2 at 3} and \autoref{eq: delta m3m4} we see that 
\begin{equation*}
\varphi_m\left(\PSL(2,\mathcal{O}_m)\right)_{(4)}\leq \varphi_m\left(\PSL(2,\mathcal{O}_m)_{(2)}\right) \leq \varphi_m\left(\PSL(2,\mathcal{O}_m)\right)_{(2)}
\end{equation*}
and
\begin{equation*}
\PSL(2,\mathcal{O}_m)_{(4)} \leq \Delta_m \leq \PSL(2,\mathcal{O}_m)_{(2)}
\end{equation*}
so $\Delta_m$ is congruence of level 4, but not principal congruence.
\end{remark}

\subsection{Virtual retracts and the RFRS condition}
Let $C$ be a hyperbolic right-angled Coxeter polytope of dimension $n>2$, say in $\H^n$. 
Let $\Gamma$ be the group generated by reflections on the faces of $C$. Define the graph $\Delta$ with a vertex for each face of $C$ and an edge between two vertices if they meet at a right angle.
\begin{equation*}
\Gamma=\langle r\in V(\Delta) : r^2=1, (rs)^2=1 \text{ if } (r,s)\in E(\Delta) \rangle 
\end{equation*}
where the generators $r$ are realized in $\Isom(\H^n)$ by reflections on the corresponding face.

Let $A$ be a subset of $V(\Delta)$ with induced subgroup $\Delta_A$ and $B$ the complement of $A$ in $V(\Delta)$ with induced subgraph $\Delta_B$. Let $\Gamma_A$ be the subgroup of $\Gamma$ generated by the elements of $A$, and similarly define $\Gamma_B$. Let $N(\Gamma_B)$ denote the normal closure of $\Gamma_B$ in $\Gamma$. 

\begin{lemma}\label{lem: retracts}
$\Gamma/N(\Gamma_B)=\Gamma_A$, that is, $\Gamma_A$ is both a subgroup and a quotient of $\Gamma$. In particular, if $F$ is a hyperface of dimension $n-k\geq2$ in $C$ and $\Gamma_F$ the reflection group of $F$ in $\Isom(\H^{n-k})$, then $\Gamma$ retracts onto $\Gamma_F$. 
\end{lemma}

\begin{proof}
Consider the group presentation 
\begin{equation*}
\Gamma/N(\Gamma_B)=\langle r\in V(\Delta) : r^2=1, (rs)^2=1 \text{ if } (r,s)\in E(\Delta), t=1 \text{ if } t\in V(\Delta_B) \rangle 
\end{equation*}
since if $t\in V(\Delta_B)$ then for any $r\in \Delta$, with $(r,t)\in E(\Delta)$, the relation $(rt)^2=1$ is equivalent in $\Gamma/N(\Gamma_B)$ to $r^2=1$, which we already had. Therefore, 
\begin{align*}
\Gamma/N(\Gamma_B) &=\langle r\in V(\Delta)- V(\Delta_B) : r^2=1, (rs)^2=1 \text{ if } (r,s)\in E(\Delta) \rangle \\
&=\langle r\in V(\Delta_A) : r^2=1, (rs)^2=1 \text{ if } (r,s)\in E(\Delta_A) \rangle \\
&= \Gamma_A.
\end{align*}

From \cite{AVS} we have the following facts:
\begin{itemize}
\item[$\circ$] Any face of a hyperbolic right-angled polytope of dimension $n>2$ is a hyperbolic right-angled polytope of dimension $n-2$.
\item[$\circ$] Every $k$-dimensional face of a hyperbolic right-angled polytope  belongs only to $n-k$ many hyperfaces. 
\end{itemize}
In particular, any $(n-3)$-dimensional face (or vertex in the case $n=3$) belongs only to 3 hyperfaces. This also means that any two $(n-1)$ faces which intersect, must intersect in an $(n-2)$ face.

Now, suppose that $F$ is a face of $C$, i.e. a hyperplane of dimension $(n-1)$. Identify $H=\H^{n-1}$ with the hyperplane of $\H^n$ containing the face $F$. Let $r$ be the reflection on $H$, and hence on $F$. Let $F'$ be a face of $C$ which meets $F$ at a right-angle. Then $r'$, reflection on $F'$, stabilizes $H$. Indeed, restricted to $H$, $r'$ is a reflection on the face of $F$ where $F'$ meets $F$. 

Suppose $F_1$ and $F_2$ both meet $F$ and intersect at an $(n-3)$-dimensional face of $C$. By the facts above, $F_1$ and $F_2$ intersect at an $(n-2)$-dimensional face at a right-angle. Therefore for the corresponding reflections $r_1$ and $r_2$, we have that $(r_1 r_2)^2=1$.

Let $S$ be the set of faces of $C$ which meet $F$ in $\H^n$ and not including $F$ itself. The map $\phi: \Gamma\rightarrow\Gamma_S$ given by
\begin{equation*}
r_i \mapsto
\begin{cases}
r_i & \text{ if } F_i\in S \\
1 & \text{ if } F_i\not\in S
\end{cases} 
\end{equation*}
is a retraction.

For a lower-dimensional hyperface (of dimension at least 2), repeat this process and compose to get a retraction from $\Gamma$ onto $\Gamma_F$.
\end{proof}

\begin{proof}[Proof of \autoref{prop: retract}]
Each of the 2-dimensional faces of the hyperbolic Coxeter polyhefrom of dimension 6 with corresponding RACG $\SO^+(F_6;\Z)_{(2)}$ is a 2-dimensional hyperbolic Coxeter polygon with RACG $\SO^+(F_2;\Z)_{(2)}$ (refer to \autoref{th: RACG}). The intersection of the image of $\Delta_m$ (after $\varphi_m$ and conjugation by $A_m$) is exaclt
\begin{equation*}
\Delta_m\cap \PSL(2,\Z) =\PSL(2,\Z)_{(2)} 
= \left\langle \begin{pmatrix} 1&2\\0&1 \end{pmatrix} ,
\begin{pmatrix} 1&0\\2&1 \end{pmatrix} \right\rangle
\end{equation*}
a free group on 2 generators.
\end{proof}

\section{Cocompact examples}\label{sec: cocompact}
In this section we first consider a family of cocompact Kleinian known to virtually embed in a RACG. We then focus on a particular group in the family and a non-Haken Kleinian group in its commensurability class. 

\subsection{A family of cocompact Kleinian groups}
Choose any positive prime $m\equiv-1\mod(8)$ and let $Q'_m$ be the quadratic form
\begin{equation}
Q'_m:=-mx_1^2 + x_2^2 + x_3 ^2 + x_4^2.
\end{equation} 
By Dirichlet's Theorem there are infinitely many such primes. For different $m$, the groups $\SO^+(Q'_m;\Z)$ are non-comensurable cocompact arithmetic lattices which all virtually embed in a RACG (see \cite[Lemma 4.6(1),\S 6]{ALR}).

Let $P'_m$ be the quadratic form $-m x_1^2 + x_2^2 + x_3 ^2 + x_4^2 + mx_5^2$. Then $P'_m = Q'_m \oplus \langle m \rangle$. The group $\SO^+(Q_m;\Z)$ is naturally a subgroup of $\SO^+(P_m;\Z)$.

\begin{lemma} \label{lem: coco first col}
The reduction modulo $m$ of an the first column of an element in $\SO^+(Q'_m;\Z)$ is $(\pm1,0,0,0)$. 
\end{lemma}

\begin{proof}
Suppose $N\in\pi_m\left(SO^+(Q'_m;\Z)\right)$. Then $N$ can be written as a block matrix
\begin{equation*}
\begin{pmatrix} N_0 & N_1 \\ N_2 & N_3 \\ \end{pmatrix} 
\end{equation*}
where $N_0$ is a $1\times 1$ matrix, $N_1$ is $1\times 3$, $N_2$ is $3\times 1$ and $N_3$ is $3\times 3$. Since $N$ must satisfy $N^t\pi_m(S_{Q'_m})N=\pi_m(S_{Q'_m})$ we have
\begin{align}
N^t\pi_m(S_{Q'_m})N & = 
\begin{pmatrix} N_0^t&N_2^t \\ N_1^t&N_3^t \\ \end{pmatrix} 
\begin{pmatrix} 0_{1\times 1} & 0_{1\times 3} \\ 0_{3\times 1} & I_{3\times3} \end{pmatrix}
\begin{pmatrix} N_0 & N_1 \\ N_2 & N_3 \\ \end{pmatrix}
\\ & = \begin{pmatrix}
N_2^tN_2& N_2^tN_3 \\ N_3^tN_2 &N_3^tN_3 \\
\end{pmatrix} \\
& \equiv \begin{pmatrix} 0_{1\times 1} & 0_{1\times 3} \\ 0_{3\times 1} & I_{3\times3} \end{pmatrix} .
\end{align}

Since $m$ is prime, $\Z/m\Z$ is a field and it must be that since $N_2^tN_2\equiv 0$, so $N_2\equiv 0_{1\times 3}$. Therefore, the the first column of $N$ has form $(a,bm,cm,dm)^t$ and must satisfy $-ma^2+(mb)^2+(mc)^2+(md)^2=-m$, equivalently $-a^2+mb^2+mc^2+md^2=-1$. This implies $a^2= 1\mod m$ and so $a= \pm1\mod m$ and the first column is $(\pm1,0,0,0)^t$ modulo $m$.
\end{proof}

\begin{remark}
An element $\mathrm{O}^+(Q'_m;\Z)$ is in $\SO^+(Q'_m;\Z)$ if the $(1,1)$-matrix entry is positive. Also, the group $\SO^+(Q'_m;\Z)$ contains elements whose first column reduces to $(-1,0,0,0)^t$. For example, when $m=7$, 
\begin{equation*}
\begin{pmatrix}
6&1&2&0 \\ -7&-2&-2&0 \\ -14&-2&-5&0 \\ 0&0&0&1
\end{pmatrix} \in \SO^+(Q'_7;\Z).
\end{equation*}
\end{remark}

Let $\Delta^m$ be the index 2 subgroup of $\SO^+(Q'_m;\Z)$ of elements whose $(1,1)$-matrix entry is equivalent to $1$ modulo $m$. 

\begin{prop}
Let $m=8k-1$ be a positive prime number. The group $\Delta^m_{(2)}$ embeds in the RACG $\SO^+(F_4;\Z)_{(2)}$.
\label{prop: congruenceembed}\end{prop}

\begin{proof}
Note first that $m=(4k)^2-(4k-1)^2$, a difference of two squares. Consider the $5\times 5$ matrix
\begin{equation}
A'_m = 
\begin{pmatrix}
4k &0&0&0& -(4k-1) \\
0&1&0&0&0 \\
0&0&1&0&0 \\
0&0&0&1&0 \\
-(4k-1) &0&0&0& 4k 
\end{pmatrix}
\label{mat: Am 5}\end{equation}
with inverse
\begin{equation}
(A'_m)^{-1} = 
\begin{pmatrix}
\frac{4k}{m} &0&0&0&\frac{4k-1}{m} \\
0&1&0&0&0 \\
0&0&1&0&0 \\
0&0&0&1&0 \\
\frac{4k-1}{m} &0&0&0&\frac{4k}{m} 
\end{pmatrix} .
\label{mat: AdI}\end{equation}

Let $S_F$ be the diagonal matrix associated to $F_4$ and $S_{P'_m}$ the symmetric matrix associated to $P'_m$:
\begin{equation}
S_F = 
\begin{pmatrix}
-1 & 0 & 0 & 0 & 0 \\
0 & 1 & 0 & 0 & 0 \\
0 & 0 & 1 & 0 & 0 \\
0 & 0 & 0 & 1 & 0 \\
0 & 0 & 0 & 0 & 1
\end{pmatrix}
\hspace{10px}\text{and}\hspace{10px}
S_{P'_m} = 
\begin{pmatrix}
-m & 0 & 0 & 0 & 0 \\
0 & 1 & 0 & 0 & 0 \\
0 & 0 & 1 & 0 & 0 \\
0 & 0 & 0 & 1 & 0 \\
0 & 0 & 0 & 0 & m
\end{pmatrix} .
\label{mat: F and Qd}\end{equation}
Then $(A'_m)^t S_F A'_m=S_{P'_m}$. Since $A'_m$ has determinant $m$, it is in $\GL(5,\Q)$. Therefore, the forms $F_4$ and $P'_m$ are equivalent over $\Q$ and thus $A'_m SO^+(P'_m;\Q) (A'_m)^{-1}=SO^+(F_4;\Q)$ with $A'_m SO^+(P'_m;\Q) (A'_m)^{-1}$ and $SO^+(F_4;\Z)$ commensurable.

By \autoref{lem: coco first col}, a matrix $N$ in $\Delta^m_{(2)}$ sits naturally in $SO^+(P'_m;\Z)$ with form
\begin{equation*}
\begin{pmatrix}
2 m a_1+1 & 2 b_1 & 2 c_1 & 2 d_1 & 0 \\
 2 m a_2 & 2 b_2+1 & 2 c_2 & 2 d_2 & 0 \\
 2 m a_3 & 2 b_3 & 2 c_3+1 & 2 d_3 & 0 \\
 2 m a_4 & 2 b_4 & 2 c_4 & 2 d_4+1 & 0 \\
 0 & 0 & 0 & 0 & 1
\end{pmatrix}  .
\end{equation*}

Then 
\begin{equation}
A'_m N (A'_m)^{-1} = 
\end{equation}
\small{\begin{equation*}
\begin{pmatrix}
32a_1 k^2+1& 8kb_1 &8kc_1 &8kd_1 &32k^2a_1-8k a_1
\\ 8 k a_2 & 2 b_2+1 & 2 c_2 & 2 d_2 & 8 k a_2-2 a_2 
\\ 8 k a_3 & 2 b_3 & 2 c_3+1 & 2 d_3 & 8 k a_3-2 a_3 
\\ 8 k a_4 & 2 b_4 & 2 c_4 & 2 d_4+1 & 8 k a_4-2 a_4 
\\ 8ka_1-32k^2a_1& 2b_1-8kb_1& 2c_1-8kc_1 &2d_1-8kd_1& -32 a_1 k^2+16 a_1 k-2 a_1+1
\end{pmatrix} 
\end{equation*}}\normalsize
is in $\SO^+(F_4;\Z)_{(2)}$.
\end{proof}

\subsection{A non-Haken example}

Commensurable with $\SO^+(Q'_7;\Z)$ and $\Delta^7_{(2)}$ is the fundamental group of a particular non-Haken hyperbolic 3-manifold mentioned in \cite[p.616]{ALR}. For this reason, we treat the case of $\SO^+(Q'_7;\Z)$ with more detail.

The arithmetic hyperbolic tetrahedral groups were studied in \cite{tetrahedral}. The tetrahedron $T_6$ of \cite{tetrahedral} has dihedral angles $\{\frac{\pi}{2},\frac{\pi}{3},\frac{\pi}{4};\frac{\pi}{2},\frac{\pi}{3},\frac{\pi}{4}\}$ and volume approximately $0.2222287320$. Let $\Gamma$ be the index-2 subgroup of orientation preserving isometries in the reflection group for this tetrahedron. It is an arithmetic Kleinian group generated by the matrices
\small{\begin{equation*}
\left\{
\begin{pmatrix} 1&0&0&0\\0&0&1&0\\0&1&0&0\\0&0&0&-1 \end{pmatrix} ,
\begin{pmatrix} 9/2&1/2&3/2&1/2\\ -7/2&1/2&-3/2&-1/2\\
-21/2&-3/2&-7/1&-3/2\\ 7/2&1/2&3/2&-1/2 \end{pmatrix} ,
\begin{pmatrix} 1&0&0&0\\0&0&0&1\\0&0&1&0\\0&-1&0&0 \end{pmatrix}
\right\} .
\end{equation*}}\normalsize

$\Gamma$ is commensurable with the group $\SO^+(Q'_7;\Z)$ generated by the matrices
\small{\begin{equation*}
\left\{
\begin{pmatrix} 1&0&0&0\\0&0&1&0\\0&1&0&0\\0&0&0&-1 \end{pmatrix} ,
\begin{pmatrix} 1&0&0&0\\0&0&0&1\\0&0&1&0\\0&-1&0&0 \end{pmatrix} ,
\begin{pmatrix} 6&1&2&0\\ -7&-2&-2&0\\
-14&-2&-5&0\\0&0&0&1 \end{pmatrix}
\right\} .
\end{equation*}}\normalsize

It turns out that $\Gamma\cap\SO^+(Q'_7;\Z)$ is actually the group $\Delta^7$. It has index 3 in $\Gamma$ and index 2 in $\SO^+(Q'_7;\Z)$ and is generated by the matrices
\small{\begin{multline*}
\left\{
\begin{pmatrix} 1&0&0&0\\0&0&1&0\\0&1&0&0\\0&0&0&-1 \end{pmatrix} ,
\begin{pmatrix} 1&0&0&0\\0&0&0&1\\0&0&1&0\\0&-1&0&0 \end{pmatrix} ,
\begin{pmatrix} 8&2&2&1\\ -14&-3&-4&-2\\ -14&-4&-3&-2\\7&2&2&0 \end{pmatrix} ,\right. \\
\left. \begin{pmatrix} 8&0&3&0\\0&1&0&0\\
-21&0&-8&0\\0&0&0&-1 \end{pmatrix} \right\} .
\end{multline*}}\normalsize

Magma shows that the reduction of $\Delta^7$ modulo 2 has order 24, so $[\Delta^7:\Delta^7_{(2)}]=24$ and $\Gamma$ has a special subgroup of index 72.

Both groups $\Gamma$ and $\SO^+(Q'_7;\Z)$ are contained in the same maximal arithmetic Kleinian group, (an image of) $\Gamma_\mathcal{O}$ where $\mathcal{O}$ is a maximal order in the invariant quaternion algebra (with notation as in \cite{tetrahedral}) with volume $\approx 0.1111143660$. It is shown in \cite{tetrahedral} that $\Gamma^{(2)}$, the subgroup of $\Gamma$ generated by all the squares, is the group $\Gamma_{\mathcal{O}^1}=P\rho_1(\mathcal{O}^1)$.

There is a non-Haken hyperbolic 3-manifold obtained by a $4/1$-Dehn filling on the once-punctured torus bundle with monodromy $R^2L^2$ (with the $RL$-factorization). It has volume $\approx 2.666744783$. Via Snap, one can check that its fundamental group, call it $\Theta$, is commensurable with $\Gamma$ because they have the same invariant arithmetic data. This group $\Theta$ is not contained in $\Gamma_\mathcal{O}$, but in some other maximal group. However, $\Theta^{(2)}$ is contained in $\Gamma_{\mathcal{O}^1}$ with index bounded above by 8. Therefore, $\Theta$ has a special subgroup of index bounded above by $8\cdot 72=576$.

\section*{Acknowledgement}
The author would like to thank Yen Duong for asking a related question. The author also thanks Alan Reid for suggesting this problem and for his generous support, guidance, and encouragement.

\bibliographystyle{alpha}
\bibliography{../../CV/NSF/mybib}





\end{document}